\documentclass[11pt,a4paper,reqno]{amsart}
\usepackage{graphicx}
\usepackage[linktocpage=true,colorlinks,citecolor=blue,linkcolor=blue,urlcolor=blue]{hyperref}
\usepackage{amssymb}
\usepackage{cite}
\usepackage{amsmath}
\usepackage{latexsym}
\usepackage{amscd}
\usepackage{amsthm}
\usepackage{mathrsfs}
\usepackage{url}
\usepackage[utf8]{inputenc}
\usepackage[english]{babel}
\usepackage{amsfonts}

\newtheorem{thm}{Theorem}[section]
\newtheorem{corr}[thm]{Corollary}
\newtheorem{lem}[thm]{Lemma}
\newtheorem{prop}[thm]{Proposition}

\theoremstyle{definition}

\theoremstyle{remark}
\newtheorem{rem}[thm]{Remark}

\numberwithin{equation}{section}
\def\Z{\mathbb Z}

\theoremstyle{definition}
\newtheorem{definition}{Definition}[section]
 \theoremstyle{remark}

\begin{document}
\title{popular differences and generalized Sidon sets}
\author{Max Wenqiang Xu}
\thanks{M.W.Xu is supported by a London Mathematics Society Undergraduate Research Bursary and the Mathematical Institute at University of Oxford.}
\address{Department of Mathematics, University College London, Gower Street, London, WC1E 6BT, United Kingdom}
\email{wenqiang.xu@ucl.ac.uk}

%\date{\today}

\maketitle
%\tableofcontents

\begin{abstract}
For a subset $A \subseteq [N]$, we define the representation function $ r_{A-A}(d) := \#\{(a,a') \in A \times A : d = a - a'\}$  and define $M_D(A) := \max_{1 \leq d < D} r_{A-A}(d)$ for $D>1$. We study the smallest possible value of $M_D(A)$ as $A$ ranges over all possible subsets of $[N]$ with a given size. We give explicit asymptotic expressions with constant coefficients determined for a large range of $D$. We shall also see how this problem connects to a well-known problem about generalized Sidon sets.
\end{abstract}

\section{Introduction}

For a positive integer $N$, denote by $[N]$ the discrete interval $\{1,2,\cdots,N\}$. For a subset $A \subseteq [N]$, we study the difference set
\[ A - A := \{a-a' : a,a' \in A\}. \]
We introduce the representation function $r_{A-A}$, defined by
\[ r_{A-A}(d) := \# \{(a,a') \in A \times A : d = a - a'\}. \]
The representation function has been investigated a lot in the literature. See, for example, \cite{Erd}, \cite{Ben}.
In this paper, we are interested in the maximum size of $r_{A-A}(d)$, as $d$ ranges over all positive integers smaller than $D$ for different thresholds $D$. More precisely, let
\[ M_D(A) := \max_{1 \leq d < D} r_{A-A}(d) \]
and $f_D(N, \alpha)$ be the smallest possible value of $M_D(A)$, as $A$ ranges over all subsets $A \subseteq [N]$ with $|A| \geq \alpha N$, i.e.,
\[ f_D(N, \alpha):= \min_{ A \subseteq [N], |A| \ge \alpha N} M_D (A).  \]

\begin{thm}\label{main-thm}  
Let $2 \le D \leq N$ be positive integers and let $\alpha \in (0,1)$. Write $D = (1+\gamma)\alpha^{-1}$ for some $-1 < \gamma $, then we have the following:
\begin{enumerate}
\item If $\gamma \le 0$, then  $f_D(N, \alpha) = 0$.
\item If $0 < \gamma \le 1 - \delta$ for some positive number $\delta$, $N \gg_\delta \alpha^{-2}$ and $D$ is sufficiently large in terms of $\delta^{-1}$, then $f_D(N, \alpha) = (1+o(1))\frac{2\gamma}{(1+\gamma)^2}  \alpha^2 N $. 
\item If $\alpha^2 N$ $\rightarrow \infty$ , $\gamma \rightarrow \infty$ and $D \alpha \rightarrow \infty$ as $N \rightarrow \infty$, then $f_D(N, \alpha) = (1+o(1))\alpha^2 N$.
\end{enumerate}
Here $\gamma = \gamma(N), \alpha = \alpha (N), D = D(N)$. They are all regarded as functions of $N$.
\end{thm}

\begin{rem}
$1)$ When $\alpha \ge 1/2$, the above conclusions can be easily checked. So we focus on $\alpha \in (0,1/2)$ from now on.\\
$2)$ A more precise statement of Theorem \ref{main-thm} $(2)$ is Theorem \ref{main-thm-1}.  \\
\end{rem}
We ask the following question related to (2) of Theorem~\ref{main-thm}.
\subsection*{Question:}
Suppose that there is a set $A$ such that $$M_D(A) = (1+o(1))\frac{2\gamma}{(1+\gamma)^2}  \alpha^2 N  ,$$ i.e., it takes the extremal value. Then, what can you say about the structure of $A$?

This is a vague inverse problem and one purpose is to investigate if the construction we give later in the Proposition~\ref{upper thm} is essentially unique or not. We believe the set we constructed is unique up to some permutations.

\subsection{Connections with Sidon sets}

We shall see that there are some connections between problems about Sidon sets and the problem we are investigating in the case $D=N$.

\begin{definition}(Sidon sets and $g$-Sidon sets)
A set of natural numbers $S$ is called a Sidon set if the equation $a+b=c+d$ has only the trivial solution $\{a,b\}=\{c,d\}$, where $a,b,c,d$ are elements of the set $S$. A set $A$ is called $g$ -Sidon set if for any integer $x$ we have:
\[r_{A+A}(x):=\# \{(a,b) \in A \times A : a+b =x\} \le g. \]
\end{definition}
Cilleruelo, Ruzsa and Vinuesa (2010) \cite{sidon} proved the following result.

\begin{thm}\label{CRV}
Define $\beta_g(N):= \max \{|A|\}$ where $ A\subseteq [N] $ and $A$ is a $g$-Sidon set.
Then,
\[ \sigma_1(g) \sqrt{gN} (1- o(1)) \le \beta_g(N) \le \sigma_2(g) \sqrt{gN} (1 +o (1)). \]
Moreover, 
\[ \lim_{g\to\infty} \sigma_1(g) = \lim_{g\to\infty} \sigma_2(g) = \sigma, ~for~some~constant~\sigma. \]
\end{thm}

We can translate our result in Theorem \ref{main-thm} into their language.

\begin{corr}
Define $\alpha_g(N):= \max \{|A|\}$ where $ A\subseteq [N] $, $r_{A-A}(x) \le g$ for all non-zero $x$.
Then,
\[ \sigma_1(g) \sqrt{gN} (1- o(1)) \le \alpha_g(N) \le \sigma_2(g) \sqrt{gN} (1 +o (1)). \]
Moreover, 
\[ \lim_{g\to\infty} \sigma_1(g) = \lim_{g\to\infty} \sigma_2(g) = 1. \]
\end{corr}

\subsection{Plan for the paper }
We first prove Theorem \ref{main-thm} $(1)$ in Section $2$ by constructing an explicit example. Theorem \ref{main-thm} $(2)$ is proved in Section $2$ by using the principle of inclusion and exclusion to prove a lower bound for $f_D(N, \alpha)$ and constructing an explicit example to attain the upper bound. To prove  the lower bound in Theorem \ref{main-thm} $(3)$ in Section $3$, we use a deterministic method by first proving results in the case of cyclic groups and then applying the results to the case of integers. \\

\subsection*{Acknowledgement}
The author would like to express his sincere gratitude to Andrew Granville and Fernando Xuancheng Shao, for carefully reading early drafts of the paper
and providing many useful comments, and also to Oleksiy Klurman who inspired the author to study this problem. Thanks to the anonymous referee for many helpful comments.

\section{Proof of Theorem \ref{main-thm} $(1)$ and $(2)$: the case $D$ is small}

We first prove Theorem \ref{main-thm} $(1)$: $f_D(N, \alpha) = 0$ when $D \le \alpha^{-1}$. 
\begin{proof}[Proof of Theorem \ref{main-thm-1} $(1)$]
We prove this fact by constructing a set $A$ with $|A| \ge \alpha N$ and $M_D(A)=0$.
Let $a:= \lfloor \alpha^{-1}\rfloor $ and $l$ is chosen to satisfy $la \le N-1 < (l+1)a$.
\[ A:=\{1, a+1, 2a+1, 3a+1, \cdots, la+1\}. \]

The density of $A$ is,

\[ \frac{|A|}{N} = \frac{l+1}{N} \ge \frac{1}{a} \ge \alpha.\]
Since $D \le \alpha^{-1}$, $D \le a$ also holds.
Since all non-zero positive differences here are at least $a$, we get $M_D(A) = 0$. This gives us $f_D(N,\alpha) = 0$.
\end{proof}

Now we prove the following theorem, which gives the result for the case when $D$ is relatively small.
\begin{thm}\label{main-thm-1}
Suppose that $D,N$ are positive integers and $\alpha \in (0,1).$ Let $D, \alpha, N$ satisfy the following relations.
\begin{enumerate}

\item $\alpha^{-1} < D \le (2-\delta) \alpha^{-1}$, for some positive $\delta < 1$ such that $D \delta >1$.
\item $N \ge 2D^3 /(D \delta -1).$
\end{enumerate}
Then

\[f_{D}(N,\alpha) = \frac{2(D\alpha - 1)}{D(D-1)}N+O(1).\]

\end{thm}

Notice that this is just a more accurate statement of Theorem \ref{main-thm} $(2)$.

\subsection{Lower bound for $f_D(N,\alpha)$}
In this section we are going to use the basic technique, the principle of inclusion and exclusion, to find the lower bound. 
\begin{prop}\label{low thm}
For any $2 \leq D \leq N$ and any $\alpha \in (0,1)$, we have
  \[ f_{D}(N,\alpha) \geq  \frac{2(D\alpha - 1)}{D(D-1)} N - \frac{2}{D}.\]

\end{prop}

\begin{proof}
Let $A \subset [N]$ be any subset with $|A| \geq \alpha N$. We need to show that
\[ M_D(A) \geq \frac{2(D\alpha - 1)}{D(D-1)} N - \frac{2}{D}. \]
For any $d \in \Z$, define the translate $A+d$ to be the set
\[ A+d := \{a+d : a \in A\}. \]
 By the principle of inclusion and exclusion, we estimate the size of the union of $A+d$ as $d$ ranges over $0 \leq d < D$ as following:\\
\begin{equation}\label{inclusion principle}
\left|\bigcup_{0 \leq d < D} (A+d) \right| \geq \sum_{0 \leq d < D}|A+d|-\sum_{0 \leq d < d' < D}|(A+d) \cap (A+d')|.
\end{equation}
Since $A+d \subseteq [N+D-1]$ for $0 \leq d < D$, the left hand side above is at most $N+D-1$. On the other hand, the first sum on the right hand side is $D|A|$. Thus we obtain
\[ \sum_{0 \leq d < d' < D}|(A+d) \cap (A+d')| \geq D|A| - (N+D-1) \geq (D\alpha-1)N - (D-1). \]
There are a total of $D(D-1)/2$ summands, and thus for some $0 \leq d < d'< D$ we have
\[ |(A+d) \cap (A+d')| \geq \frac{2}{D(D-1)} ((D\alpha-1)N - (D-1)) = \frac{2(D\alpha-1)}{D(D-1)}N - \frac{2}{D}. \]
The conclusion follows by noting that $|(A+d) \cap (A+d')|$ is precisely $r_{A-A}(d'-d)$, and $1 \leq d'-d < D$. 
\end{proof}

\subsection{Upper bound for $f_D(N,\alpha)$}

\begin{prop}\label{upper thm}
Let $2 \leq D \leq N$ and $\alpha \in (0,1)$ be parameters satisfying
\begin{equation}\label{2alpha}
 D(2-D\alpha) > 1 
\end{equation}
and
\begin{equation}\label{Nalpha}
 N \geq \frac{D(D-1)(2D+1)}{D(2-D\alpha)-1}. 
\end{equation}
Then
\begin{equation}\label{S2 upper}
f_{D}(N,\alpha) \leq  \frac{2(D\alpha-1)}{D(D-1)} N + 4.
\end{equation}
\end{prop}

\begin{rem}
For example, if $\alpha = o(1)$ then these assumptions hold when $D \leq (2-c)/\alpha$ for any absolute constant $c>0$ such that $Dc >1$ and $N$ is much larger than $\alpha^{-2}$.
\end{rem}

\begin{proof}
We construct one subset $A$ with particular structure such that it has density at least $\alpha$ and $M_D(A)$ attains the upper bound in \eqref{S2 upper}.

We define a set $B \subseteq \{1,2,\cdots, D^2\}$ as following:
\begin{equation}
B := \{D, 2D, 3D,\cdots, D^2\} \cup \{1, D+2, 2D+3, \cdots, (D-2)D+D-1\}
\end{equation}
so that $|B| = 2D-1$. Let $M = \lfloor N/D\rfloor$ and let $k = \lceil (\alpha N - M)/(D-1) \rceil$ so that
\[ k \leq \frac{\alpha N - M}{D-1} + 1 \leq \frac{\alpha N - N/D}{D-1} + 2 = \frac{D\alpha-1}{D(D-1)}N + 2. \]
By the lower bound on $N$, we have
\[ kD \leq \frac{D\alpha-1}{D-1}N + 2D \leq \frac{N}{D}-1 \leq M. \]
Thus we may define $A_1,A_2 \subseteq [N]$ by
\begin{equation}
A_1:=\{x+iD^2 : x\in B, i=0,1,2,\cdots, k-1)\}, \ \ A_2 := \{jD : kD < j \le M\}, 
\end{equation}
and then take $A = A_1 \cup A_2$. Hence
\[ |A| = k|B| + (M-kD) = kD+M-k \geq \alpha N,  \]
where the last inequality follows by our choice of $k$.

We only need to show that  $r_{A-A}(d) \leq 2k$ holds for any $1 \leq d < D$ to complete our proof. For any $1 \leq d < D$, if $d = a-a'$ for some $a,a' \in A$, then we must have $a,a' \in A_1$ since any two distinct elements must have difference at least $D$ if at least one of them is in $A_2$. Writing $a = x+iD^2$ and $a' = x' + i'D^2$ with $x, x' \in B$ and $0 \leq i,i' < k$, we first observe that the possible values of $|i-i'|$ can only be $0$ or $1$. \\
If $|i-i'|\ge 2$, then by triangle inequality we have 
\[\left|(x+iD^2) - (x'+i'D^2)\right| \ge |i-i'|D^2 - \max_{x,x' \in B}|x-x'| \ge D^2 +1 > D.\]
If $|i-i'| = 1$, by considering the structure of $B$ we have if $\{x,x'\} \neq \{D^2,1\}$, then 
\[ \left|(x+iD^2) - (x'+i'D^2)\right| \ge D. \]
So one must have $\{x, x'\} = \{D^2, 1\}$ and $d=1$ if $|i - i'|=1$. As for $d=1$, if $i=i'$, then one must have $x=(D-2)D+D-1$ and $x'= (D-1)D$. So $r_{A-A}(1) = k + (k-1) < 2k$. As for $d \ge 2$, we showed that $i$ and $i'$ must satisfy $i=i'$ which gives $d=x-x'$. And the number of possible choices for $x,x'$ in $ B$ is not larger than $2$ in this case. So $r_{A-A}(d) \le 2k$ for all $d \neq 1$. In conclusion, the result holds for all $1\le d <D$.  
\end{proof}
At the end of this subsection, we give an explanation for why the principle of inclusion and exclusion works.
Since for any $y$ $\in$ $(A + i) \cap (A + j) \cap (A + k)$, where $ i, j, k \in \{0,1,2,\cdots, D-1\} $, there exists $x \in A$, such that $x+i,x+j,x+k$ are all in $A$, i.e., they lie in an interval with length $D-1$. But this is impossible in our constructed example, where any interval with length $D-1$ has at most $2$ elements instead of $3$.\\
There is no point in the intersection of the $3$ sets, so it is even impossible for $4$ or more sets.
Hence the values of all the moments with order greater than two must be zero.\\

\subsection{Conclusion for $f_D(N,\alpha)$ with $ \alpha^{-1} < D \le (2-\delta)\alpha^{-1}$}
In this part, we combine the two propositions before to prove Theorem \ref{main-thm-1}.

\begin{proof}[Proof of Theorem \ref{main-thm-1}]
The lower bound in Proposition~\ref{low thm} and the upper bound in Proposition~\ref{upper thm} together give us the result in Theorem \ref{main-thm-1}. We only need to check that the assumptions \eqref{2alpha} and \eqref{Nalpha} are satisfied.\\
Since there exists $\delta > 0$ such that $D\le (2-\delta)\alpha^{-1}$, the assumption $(1)$ in Theorem \ref{main-thm-1} holds by the fact that 
\[ D(2-D\alpha) \ge D\delta >1. \]
 Now we check that assumption \eqref{Nalpha} holds. Notice that
\[ \frac{D(D-1)(2D +1)}{D(2-D\alpha)-1} \le \frac{2D^3}{D\delta -1}.\] 
From assumption $(2)$ in Theorem \ref{main-thm-1}, we know \eqref{Nalpha} holds.
\end{proof}

\begin{rem}\label{Idea}
From the whole pattern, we shall see that $D=\lfloor \alpha^{-1} \rfloor + 1$ is a crucial case in the problem. We state this result separately here, which has been covered in Theorem \ref{main-thm-1}. To make the form nicer, we state the result in a special case that $\alpha^{-1}$ is an integer.
\[  f_{D}(N,\alpha)=   \frac{2\alpha^3}{1+\alpha}N + O(1).    \]
\end{rem}
This result can also compare to \cite[Lemma 2.3]{Ole}. We state their result here.
\begin{lem}\label{Ole}
Suppose $S \subset [H, H + N]$ and $|S| = \alpha N$. Here $[H, H+N]$ is an integer interval. Then there exists a positive integer $d$, such that $d \le 1/ \alpha$ and 
\[r_{S-S}(d) \ge \frac{ \alpha^3}{2}N. \] 
\end{lem} 

In Theorem \ref{main-thm}, there is one interesting phenomenon when $\gamma$ is relatively small. Precisely, if $\gamma =O(\alpha)$ then we have $f \asymp \alpha^3 N$, e.g., the case in remark \ref{Idea}.

\section{Proof of the lower bound in Theorem \ref{main-thm} $(3)$}

In this section, we use Fourier analytic tools to prove Theorem \ref{Lingbo 1} which gives the lower bound of $f_D(N,\alpha)$ for a wide range of $D$. The argument is similar to the proof of \cite[Theorem 1.1]{Ben}.

\begin{thm}\label{Lingbo 1}
Let $2 \le D \le N$, then 
\begin{equation}\label{Increasing}
f_D(N,\alpha) \ge \frac{\alpha^2N^2}{N +D} - \frac{\alpha N +1}{D}.
\end{equation}
In particular, if $D\alpha \rightarrow \infty $, we have 
\begin{equation}\label{Finally}
f_D(N,\alpha) \ge (1+o(1))\alpha^2 N.
\end{equation}
\end{thm}

\begin{proof}
Let $A$ be any subset of $[N]$ with density at least $\alpha$. We regard $A$ as a subset of $\mathbb{Z}/(N+D) \mathbb{Z}$ in the natural way.
To fix our notation, we define the convolution of $f,g$ where $f,g: G \rightarrow \mathbb{C}$ are two functions on the abelian group $G$.
\[f * g (x) := \sum_{y \in G} f(y)  \overline {g(x-y)}.\]
And we use our traditional notation to get that 
\[1_A * 1\textordmasculine_A(x) = \# \{ (a,a') \in A \times A: a-a'=x\},\]
where $f \textordmasculine(x):= f(-x)$.\\
We use $M_D(A)$ to denote the quantity the same as in the case of integers before. It is not true that for general $x \in \mathbb{Z}/(N+D) \mathbb{Z}$, $1_A * 1\textordmasculine_A(x) \le M_D(A)$ in modular version. However, for any $x$ with $0<|x|\le D$, in the modular version we still have
\[1_A * 1\textordmasculine_A(x) \le M_D(A).\]
Let $I$ be the characteristic function of the set $\{1,2,\cdots, D\}$. Then we estimate the following quantity.
\[E:= \sum_{x \in \mathbb{Z}/(N+D) \mathbb{Z}}1_A * 1\textordmasculine_A(x) I*I\textordmasculine(x).\]
We first notice that each summand is zero when $|x| > D$. Combining this and what we discussed above, we have 
\[E \le |A| D + \sum_{0 < |x| \le D} 1_A * 1\textordmasculine_A(x) I*I\textordmasculine(x). \]
The right hand side can be further upper bounded by 
\[|A|D +  M_D(A) D^2. \]
On the other hand, we use Plancherel's identity to get
\[ E = (N+D) \sum_{r}|\hat 1_A(r)|^2|\hat I(r)|^2 \ge \frac{|A|^2 D^2}{N+D}. \]
The above two inequalities give us the bound for $M_D(A)$ for any subset $A$ with $|A| \ge \alpha N$. For any $A$ such that can make $M_D(A) = f_D(N,\alpha)$, we must have $|A|= \lceil \alpha N \rceil$. By substituting this into our expression, we get the lower bound. 
Notice that if $D\alpha \rightarrow \infty$, the lower bound in \eqref{Increasing} has main term $\frac{\alpha^2 N^2}{N + D}$. This main term is a decreasing function of $D$. Therefore, under the assumption that $D \alpha \rightarrow \infty$, it is better to use the lower bound in \eqref{Increasing} with $D=o(N)$ and it gives \eqref{Finally}. Since function $f_D(N, \alpha)$ is an increasing function of $D$, with the assumption that $D \alpha \rightarrow \infty$, \eqref{Finally} holds no matter $D = o(N)$ holds or not.
\end{proof}

\section{Proof of the upper bound in Theorem \ref{main-thm} $(3)$}
In this section we study the upper bound of $f_D(N, \alpha)$ with given density $\alpha$ and $N$. We consider the case $D= N$, i.e., we can regard it as restriction free on the range of differences of any two distinct elements in $A$. 

Let $f_N(\alpha)$ simplify the notation $f_D(N,\alpha)$ we defined before in the case $D=N$.
\begin{thm}\label{Lingbo 2}
Let $\alpha \in (0,1)$. If  $\alpha^2 N \rightarrow \infty$ as $N \rightarrow \infty$, then 
\[f_N(\alpha) = (1+o(1))\alpha^2 N. \]
\end{thm}

In the above condition, we need to prove the result for $\alpha < 1- \epsilon$ for any positive $\epsilon$ smaller than one. Aiming for a simple expression and without loss of generality, we will only prove the case $\alpha<1/2$ for the upper bound.\\
The upper bound here can be proved by the deterministic method. It also can be proved by the probabilistic method, however, a shortcoming of probabilistic method here is we can not always avoid the non-essential needed restriction on the range of density $\alpha$ that it be much larger than $  N^{-1/2} (\log N)^{-1/2}$; at least the author does not know how to avoid this. So in this paper, we use the deterministic approach to prove asymptotic formulas for the function $f$ with loose restriction on $\alpha$.  Before we prove the upper bound in the case of integers, we first prove the following theorem which gives us the results in cyclic groups and would imply the results in $\mathbb{Z}$.

\begin{thm}\label{Z}
Let $N = p^2 s$ for some positive integer $s$ and prime number $p$. Let $\mathbb{Z}/N\mathbb{Z}$ be a cyclic group and $\alpha \in (0, 1)$. 
We define 
\[f_N(\alpha) : =  \min_{A \subseteq  \mathbb{Z}/N\mathbb{Z}, |A| \ge \alpha N} \max_{d \neq 0} r_{A-A}(d).\]
If $p, s, p \alpha \rightarrow \infty $, then  
\[f_N(\alpha) = (1+o(1))\alpha^2 N .\]
\end{thm}

On the other side, there is a similar result proved in generalized Sidon sets. See \cite[Theorem 1.6]{sidon}. We state it by using our notation.
\begin{thm}
Let $\mathbb{Z}/N\mathbb{Z}$ be a cyclic group and $\alpha \in (0, 1)$. 
We define 
\[g_N(\alpha) : =  \min_{A \subseteq  \mathbb{Z}/N\mathbb{Z}, |A| \ge \alpha N} \max_{d } r_{A + A}(d).\]
Then as $N$ tends to infinity,
\[  \liminf_{N \to \infty} \frac{g_N(\alpha)}{\alpha^2N }= 1,\]
Here\[ r_{A+A}(d) = \{(a,a') \in A \times A : d = a + a'\}. \]
\end{thm}

Similarly, we have the following in the case of difference sets.
\begin{equation}\label{inf}
  \liminf_{N \to \infty} \frac{f_N(\alpha)}{\alpha^2N }= 1.
\end{equation}
\begin{proof}[Proof of \eqref{inf}]
The lower bound of $f_N(\alpha)$ can be regarded as a corollary of Theorem \ref{Lingbo 1}. We prove it by a simple counting argument here. We still use $M_N(A)$ to denote the maximum value of $r_{A-A}(x)$ for all possible non-zero $x$ in the cyclic group context. There are $|A|^2$ number of pairs $(a,a') \in A \times A$ but $|A|$ of them satisfy $a=a'$. So we have the following 
\[|A|^2- |A| \le (N-1)M_N(A).\]
By noting the fact that $|A| \ge \alpha N$, and letting $N$ tend to infinity, we have the following when $N$ is sufficiently large.
\[f_N(\alpha) \ge \alpha^2 N - 1.\]
This gives the lower bound of $f_N(\alpha)$. Combining this with Theorem \ref{Z}, we complete the proof. 
\end{proof}

To prove Theorem \ref{Z}, we first do the construction  in $\mathbb{Z}/p \mathbb{Z} \times \mathbb{Z}/p \mathbb{Z}$ where $p$ is a prime and then move on to $\mathbb{Z}/q\mathbb{Z}$ with a special type of $q$. The argument here is similar to \cite[Section 3, 4]{sidon}. Finally, the results in the cyclic group will give us the results in $\mathbb{Z}$ that we are aiming for.

\subsection{Construction in   $\mathbb{Z}/p \mathbb{Z} \times \mathbb{Z}/p \mathbb{Z}$}
We prove that there is a set $A$ such that $r_{A-A}$ has a relatively small upper bound and $|A|$ is reasonably large in  $\mathbb{Z}/p \mathbb{Z} \times \mathbb{Z}/p \mathbb{Z}$.

\begin{prop}\label{prop 1}
Given a positive integer $k$, for any odd prime number $p \ge 2k$ and any $(a,b) \not\equiv (0,0) \pmod p$, there is a set $A \subseteq \mathbb{Z}/p \mathbb{Z} \times \mathbb{Z}/p \mathbb{Z}$ with $|A| = kp - k +1$ such that $r_{A-A} (a,b) \le \lfloor k^2 + 7k^{7/4} \rfloor$.
\end{prop}

In the context $\mathbb{Z}/p \mathbb{Z} \times \mathbb{Z}/p \mathbb{Z}$, we do our construction  by pasting several disjoint sets together. The main techniques will be used are properties of quadratic equations.\\
First we define the set which plays an important role in this section.
For $ u \not\equiv 0 \pmod p $, define 
\begin{equation}\label{Au}
A_u:= \{(x, x^2/u): x \in \mathbb{Z}/p \mathbb{Z}\}.
\end{equation}

For any $(a,b) \in \mathbb{Z}/p \mathbb{Z} \times \mathbb{Z}/p \mathbb{Z} $, we use the following notation to denote the representation function. For $uv \not\equiv 0 \pmod p$, define
\[ r_{A_u - A_v}(a,b) : = \# \{ ((x, x^2/u), (y, y^2/v)) : a \equiv x - y, b \equiv x^2/u- y^2/v \pmod p     \}.\]
This is exactly the number of solutions to the equations 
\[ 
\begin{cases}
 a \equiv x - y \pmod p \\
 b \equiv \frac{x^2}{u} -  \frac{y^2}{v} \pmod p.
\end{cases}
\]

In the next lemma we study the possible values of $r_{A_u - A_v}(a,b)$, where $(a,b) \not\equiv (0,0) \pmod p$.

\begin{lem}\label{dis}
Suppose $r_{A_u - A_v}(a,b)$ is defined as above, $p$ is an odd prime and $(a,b) \not\equiv (0,0) \pmod p$.
\begin{enumerate}
\item If $u \equiv v \pmod p$, then $r_{A_u - A_v}(a,b) \le 1$.
\item If $u-v \equiv u'-v' \pmod p $ and $(\frac{uvu'v'}{p})=-1$, then $r_{A_u - A_v}(a,b) + r_{A_{u'}- A_{v'}}(a,b) =2$ for all $(a,b) \not\equiv (0,0) \pmod p$.
\end{enumerate}

The symbol $(-)$ we use in $(2)$ is Legendre symbol.
\end{lem}

\begin{proof}
1) In the case $u \equiv v \pmod p$, it is clear that if $a \equiv 0 \pmod p$ then there is no solution since we required that $(a,b) \not\equiv (0,0) \pmod p$. Otherwise we have $x+y \equiv uba^{-1} \pmod p$ which leads to a unique solution with the fact that $x-y \equiv a \pmod p$.\\
2) In the case that  $u-v \equiv u'-v' \pmod p$ and $(\frac{uvu'v'}{p})=-1$, we have $u \not\equiv v \pmod p$ and $u' \not\equiv v' \pmod p$. It means that we are studying proper quadratic equations in this case. Substituting $y \equiv x-a \pmod p$ into $b \equiv x^2/u- y^2/v \pmod p$, we have the quadratic equation
\[(u-v)x^2 - 2aux + ua^2 + buv \equiv 0      \pmod p          \]
and the discriminant is 
\[  \Delta \equiv 4uv (a^2- (u-v)b)                       \pmod p .    \]
The number of solutions can be written as 
\[ 
r_{A_u - A_v}(a,b) =
\begin{cases}
1 &\text{if }  (\frac{\Delta}{p}) = 0.\\
2 &\text{if }  (\frac{\Delta}{p}) = 1.\\
0 &\text{if }  (\frac{\Delta}{p}) = -1.
\end{cases}
\]

This can be simplified as 
\[r_{A_u - A_v}(a,b) = 1+ \Big(\frac{\Delta}{p}\Big).\]
To calculate the sum of $r_{A_u - A_v}(a,b)$ and $ r_{A_u'- A_v'}(a,b) $, we need to calculate the corresponding sum $(\frac{\Delta}{p}) + (\frac{\Delta'} {p})$.
Notice that the product of these two is
\[ \Big(\frac{\Delta}{p}\Big) \Big(\frac{\Delta'} {p}\Big) = \Big(\frac{16uvu'v'((u-v)b - a^2)((u'-v')b - a^2)}{p}\Big)  = - \Big(\frac{((u-v)b - a^2)^2}{p}\Big).\]
In the last step we used the fact that $u-v \equiv u'-v' \pmod p$ and $(\frac{uvu'v'}{p})=-1$. Now there are two possibilities. If $(u-v)b - a^2 \equiv 0$, then both $(\frac{\Delta}{p})$ and $(\frac{\Delta'}{p})$ equal to zero. If $(u-v)b - a^2 \not\equiv 0$, then the product of $(\frac{\Delta}{p})$ and $(\frac{\Delta'}{p})$ equals to $-1$. These two cases give us the same result
\[ \Big(\frac{\Delta}{p}\Big) + \Big(\frac{\Delta'} {p}\Big) = 0, \]
which means that $r_{A_u - A_v}(a,b) + r_{A_u'- A_v'}(a,b) = 2$.
\end{proof}

In the next lemma, the main technique we use is Weil's Theorem.
\begin{lem}\label{S}
Suppose $0 \le n \le k-1$, $1 \le |l| \le k-1$, $1 \le i,j \le k$ and $p$ is an odd prime number. Let $S_n$ define the sum
\[S_n: = \sum_{1\le |l| \le k-1}\left|\sum_{i-j=l} \Big(\frac{(n+i)(n+j)}{p}\Big)\right|.\]
Then we have the following upper bound for the sum of $S_n$
\[       \sum_{0 \le n \le p-1} S_n \le  \sqrt{2p^2k^2 (2k-1)+ 8p^{3/2}k^4}.  \]
\end{lem}

To estimate the sum, we first  use the Cauchy inequality to rewrite the sum which roughly gives us a function with  power $4$ in the numerator of Legendre symbol, and then apply the Weil's Theorem.
\begin{proof}
First we use Cauchy inequality to get the following estimation.
\[ \sum_{0 \le n \le p-1} S_n = \sum_{n,l} \left|\sum_{i-j=l}\Big(\frac{(n+i)(n+j)}{p}\Big)\right|  \le \sqrt{ 2kp \sum_{n,l} \Big( \sum_{i-j=l}\Big( \frac{(n+i)(n+j)}{p} \Big) \Big)^2 }. \]
The right hand side above can be written as 
\[ \sqrt{ 2kp \sum_{i-j = i' -j'}  \sum_{n}\Big( \frac{(n+i)(n+j)(n+i')(n+j')}{p} \Big)  }.\]
Let $f(n)=(n+i)(n+j)(n+i')(n+j')$. We estimate the quantity above by using Weil's Theorem.\\
1) If $f(n)$ is a constant multiple of a square, then we bound the value $\Big(\frac{f(n)}{p}\Big)$ trivially by $1$. The number of quadruples $(i,j,i',j')$ that can make $f(n)$ a constant multiple of a square is bounded by $k(k-1)+k^2$. The reasons are the following. First notice that if the restriction on the quadruple is $i+j'=i'+j$, then we have two possible cases, either $i=i'$ and $j=
j'$, or $i=j$ and $i'=j'$.
These two cases both give $k^2$ choices but there is an overlap that all of them are equal, i.e., we need to subtract $k$ choices. In total, we have the following bound
\[\sum_{f~ is ~square} \sum_{n} \Big( \frac{f(n)}{p}\Big) \le pk(2k-1).                                   \]
2) If $f(n)$ is not a constant multiple of a square, we can use Weil's Theorem. For each quadruple $(i.j,i',j')$, we have the following, 
\[ \sum_{0 \le n \le p-1} \Big( \frac{f(n)}{p}\Big) \le (deg f) \sqrt{p}= 4 \sqrt{p}. \]
The total number of quadruples under the restriction $i-j=i'-j'$ is bounded by $k^3$. We also use this bound as our bound for the number of quadruples which make $f$ satisfy Weil's Theorem. Then we have the following

\[\sum_{f~isn't~a~square} \sum_{n} \Big( \frac{f(n)}{p}\Big) \le 4k^3 \sqrt{p}.                              \]
Combining case $1)$ and case $2)$,  we have 
\[       \sum_{0 \le n \le k-1} S_n \le  \sqrt{2pk ( p k(2k-1) + 4k^3 \sqrt{p})} =   \sqrt{2p^2k^2 (2k-1)+ 8p^{3/2}k^4}. \]
\end{proof}

Now we will use the above two lemmas to prove Proposition~\ref{prop 1}.

\begin{proof}[Proof of Proposition \ref{prop 1}]
First we point out how the set $A$ looks. Let $A_u$ be the same as what we defined in \eqref{Au}, and $A$ be  defined as follows:
\[ A : = \bigcup_{n+1 \le u \le n +k}  A_u.\]
The crucial point is to find a suitable $n$ which makes $A$ be the set we are searching for.\\ 
First notice that the size of $A$ is independent of $n$. And the size of $A$  is $k(p-1)+1$ since $A_u$ and $A_v$ only intersect at point $(0,0)$ for any $u \neq  v$.
Next we study the representation function $r_{A-A}$. From the definition of $A$, we can have a trivial upper bound. For any $(a,b)$ non-zero, we have
\[ r_{A-A}(a,b) \le \sum_{n+1 \le u,v \le n+k} r_{A_u - A_v} (a, b).\]
To study the right hand side above, we first parametrize the variables above. Write $u = n + i$, $v= n +j$, then $ i,j \in \{1,2,3, \cdots,k\}$. Let $l = i- j = u - v$, we have the following,
\[ |l| \le k -1.      \]  
Fixing the value of $l$, we have $k - |l|$ pairs of $(i, j)$, i.e., $k - |l|$ pairs of $(u,v)$. Now we are going to use Lemma \ref{dis}.\\
1) If $l = 0$, from the first part of the lemma, we have the bound 
\[\sum_{u=v}r_{A_u-A_v}(a,b) \le k .\]
2) If $|l| \ge 1$, we consider the value of $\Big( \frac{uv}{p}  \Big)$ for each pair, $(u,v)$. Let $\alpha(l) : = \# \{(u,v): u- v = l,\Big( \frac{uv}{p}  \Big) = 1 \}$, 
$\beta(l) : = \# \{(u,v): u- v = l,\Big( \frac{uv}{p}  \Big) = -1 \}$. Then we have the following fact
\[ \alpha(l) + \beta(l) = k - |l|,~~\alpha(l)- \beta(l) = \sum_{u-v= l} \Big( \frac{uv}{p}\Big).  \]
Our strategy is to match pairs into pairs according to their values under Legendre symbol. Precisely, if two pairs have different signs, we can match them and then use the second part in Lemma \ref{dis} to estimate the sum. Since the values of $\alpha(l)$ and $\beta(l)$ are not necessarily equal, for the rest of the pairs which have not been  matched with other pairs, we trivially bound their function values by $2$. Hence, we have 
\[\sum_{1 \le| l| \le k} r_{A_u - A_v}(a,b) \le 2 \min \{\alpha(l), \beta(l)\} + 2 (\max \{\alpha(l)  ,\beta(l)\} - \min \{\alpha(l), \beta(l)\}). \]
The above actually gives us the result $ 2 \max \{\alpha(l)  , \beta(l)\}$, i.e., $\alpha(l)+ \beta(l) + |\alpha(l) - \beta(l) |$,  which can be further written as 
\[ k - |l| +  \left|\sum_{u-v= l} \Big( \frac{uv}{p}\Big) \right|.\]
Summing the above expression over $l$ where $1 \le |l| \le k-1$ and considering the obvious bound that we have showed for $l =0$ in case $1)$, we have the following
\[\sum_{| l| \le k-1} r_{A_u - A_v}(a,b) = k +\sum_{1 \le |l| \le k-1} r_{A_u - A_v}(a,b) = k^2 + \sum_{1 \le| l| \le k-1} \left|\sum_{u-v= l} \Big( \frac{uv}{p}\Big) \right|.\]
Notice that the second part on the right hand side above is exactly what we defined in Lemma \ref{S}, i.e., the upper bound now is 
\begin{equation}\label{express}
k ^2 + S_n.
\end{equation}
Since $uv \not \equiv 0 \pmod p$, we have the following upper bound for $n$, 
\[ n+ k \le p-1 \iff n \le p- k - 1.\]
We define $n^*$ to be the smallest integer index among $\{0,1,2,\cdots,p - k - 1\}$ such that
\[S_{n^*} = \min \{S_n: 0\le n \le p- k-1 \}. \]
Then we can see that $n^*$ is uniquely defined. By using Lemma \ref{S}, we have
\[   S_{n^*} \le  \sum_{ 0 \le n \le p- k- 1} \frac{S_n}{  p-k }  \le    \sum_{ 0 \le n \le p- 1} \frac{S_n}{ p- k}     =     \frac{ \sqrt{2p^2k^2 (2k-1)+ 8p^{3/2}k^4} }{p-k}.\]
This bound is smaller than
\[\frac{ 2\sqrt{p^2k^3 + 2p^{3/2}k^4} }{p-k}.    \]
Write $p = \lambda k$, where $\lambda \ge 2$, then the above expression is a decreasing function with respect to $\lambda$. So we can further get a bound
\begin{equation}\label{S2}
 S_{n^*} < 7k^{7/4}.
\end{equation}
By substituting \eqref{S2} into \eqref{express}, we get the upper bound 
\[k^2 + 7 k ^{7/4}.\]  
Hence, we have found a  set $A$ with $kp-k+1$ elements and $r_{A-A}$ is upper bounded by $\lfloor k^2 + 7 k^{7/4} \rfloor$. The set $A$ can be represented as following 
\[A =\bigcup_{n^*+1 \le u \le n^* +k}  A_u.\]
\end{proof}

\subsection{Construction in   $\mathbb{Z}/q \mathbb{Z}$  with $ q = p^2s$}

We construct a large set $A'$ in a certain cyclic group. The strategy is to project the set $A$ constructed in $\mathbb{Z}/p \mathbb{Z} \times \mathbb{Z}/p \mathbb{Z}$ onto $\mathbb{Z}/q \mathbb{Z}$.

\begin{lem}\label{thm 2}
Suppose that $A$ is a subset of $\mathbb{Z}/p\mathbb{Z} \times \mathbb{Z}/p\mathbb{Z}$  with size $m$, such that $r_{A-A}$ is upper bounded by $h$. Then there exists a subset $A'$ in $\mathbb{Z}/p^2s \mathbb{Z}$ with size $|A'| = ms$ such that $r_{A-A}$ is upper bounded by $ h(s+1)$.
\end{lem}

\begin{proof}
Notice that any element in $A$ has two coordinates and we choose the natural way to write them as integers in $[0, p-1]$. Now we use set $A$ to define set $A'$ as following
\[A' = \{ x= a+ cp + bsp : (a,b) \in A, 0\le c \le s-1 \}.\]
From the definition, we have $|A'| = |A| \times s = ms$. We show that $r_{A'-A'}$ is bounded by $h(s+1)$. In other words, we need to study the representation function $r_{A-A}(x)$ for non-zero $x \in \mathbb{Z}/p^2s \mathbb{Z}$ . Given $x = a+ cp + bsp$, then $a,b,c$ are fixed. Write $x$ as the difference of two elements in $A'$,
\begin{equation}\label{1}
   a+ cp+ bsp \equiv (a_1 +c_1p + b_1sp) - (a_2 + c_2 p + b_2 sp)  \pmod {p^2s}.     
\end{equation}
From the above we can get 
\[ a \equiv a_1- a_2   \pmod p. \]
There are two possible values of the difference $a_1 - a_2$. Either $ a_1 - a_2 = a$ or $a_1 - a_2 =a -p$. For convenience, we write  $a_1- a_2 = a - \delta p$ where $\delta = 0,1$. This also tells us that $\delta = (a+a_2-a_1)/p$ is fixed once $a_1, a_2$ are given.\\
Now we substitute this expression into \eqref{1}. We can derive that 
\begin{equation}\label{2}
c+bs \equiv - \delta + c_1- c_2 + (b_1 -b_2)s \pmod {ps}.
\end{equation}
Again, we have the following from \eqref{2}
\begin{equation}
c \equiv - \delta + c_1 - c_2 \pmod s.
\end{equation}
 Thus we have $c = - \delta + c_1 - c_2 $ or $c = - \delta + c_1 - c_2 + s$. We write this as: 
\begin{equation}\label{3}
c -\eta s =  - \delta + c_1 - c_2 ~ where ~ \eta = 0,1.
\end{equation}
We substitute this into \eqref{2} and we get  
\[ b \equiv - \eta + b_1 - b_2 \pmod p. \]
Now we have the expression of the element $(a,b) $ in $A$, 
\[(a,b)= (a_1, b_1)- (a_2, b_2) + (\delta, -\eta).\]
We also have the following form 
\[ (a,b) \equiv (a_1, b_1)- (a_2, b_2) + (0, -\eta) \pmod p.\]
We have that $r_{A-A}$ is upper bounded by $h$. Once we have fixed $ \eta$, the number of choices for $a_1,a_2,b_1,b_2$ is no more than $h$, and $\delta $ is also determined.\\
So now we only need to consider the number of choices for $c_1,c_2$ once $a_1,a_2,b_1,b_2, \delta$ are fixed. We come back to Equation \eqref{3} and discuss the possible cases in terms of the values of $\eta$.\\
1) If $\eta = 0$, we have $c_1 = c_2 + c +\delta \ge c$, which leads to at most $s-c$ possibilities.\\
2) If $\eta = 1$, we have $c_1 = c - s +\delta + c_2 \le c -s +1 +s-1  = c$ , which leads to at most $c+1$ possibilities.\\
In total, the above will give us at most $h(s-c+c+1)= h(s+1)$ possibilities.
\end{proof}

\subsection{Conclusion for cyclic group}

First we combine Proposition \ref{prop 1} and Lemma \ref{thm 2} to get the following.

\begin{prop}\label{zan}
For any positive integers $k, s$, and any positive odd prime integer $p \ge 2k $, there is a set $A \subseteq \mathbb{Z}/p^2s\mathbb{Z}$ with size $(kp-k+1)s$ such that  $r_{A-A}$ is upper bounded by $\lfloor k^2+7 k^{7/4} \rfloor (s+1)$.
\end{prop}
Now we can prove Theorem \ref{Z}.
\begin{proof}[Proof of Theorem \ref{Z}]

The lower bound is proved in the proof of \eqref{inf}. Now we prove the upper bound. 
For any positive $\alpha <1/ 2$, choose $k$ to be the integer such that $k = \lceil p \alpha \rceil+1$. Since $p \alpha$ tends to infinity, 
\begin{equation}\label{kp}
k/p = (1+o(1))\alpha.
\end{equation}
From Proposition \ref{zan}, we know there exists a subset $A \subseteq \mathbb{Z}/N \mathbb{Z}$ with size  $(kp-k+1)s$ such that the representation function $r_{A-A}$ is upper bounded by $\lfloor k^2+7 k^{7/4} \rfloor (s+1)$. First we consider the density of $A$,
\begin{equation} \label{d}
\frac{(kp- k +1)s}{|N|} \ge \frac{kp-k+1}{p^2}\ge \frac{k-1}{p}\ge \alpha.
\end{equation}
Then we consider the expression of  the upper bound for $r_{A-A}$,
\[\lfloor k^2+7 k^{7/4} \rfloor (s+1) = (1+o(1))(k^2/p^2)N = (1+o(1)) \alpha^2 N.\]
The expression is valid since we let $s$ and $k$ tend to infinity and Equation \eqref{kp} holds. So the above has showed that for a special type of integers, $N$, $f_N(\alpha)$ is upper bounded by $(1+o(1)) \alpha^2 N$. This completes the proof.

\end{proof}

Now, we can easily prove Theorem \ref{Lingbo 2}
\begin{proof}[Proof of Theorem \ref{Lingbo 2}]

For any $N$ in the form $N= p^2s$ for some suitable prime $p$ and integer $s$, let  set $A \subseteq \mathbb{Z}/N\mathbb{Z}$ be the set in Theorem \ref{Z}. We can also view it as a subset of $[N]$ in the natural way. 
To distinguish, we use $r'_{A-A}$ to denote the representation function in $\mathbb{Z}/N\mathbb{Z}$. Then a natural inequality is 
\[ r_{A-A}(x) \le r'_{A-A}(x) \le (1+o(1)) \alpha^2 N. \] 
For $N$ not in the form $N = p^2 s$, we choose $p =o(\sqrt{N})$ to be a prime number in the interval $[P,2P]$ for some integer $P$. This is valid once $P$ is sufficiently large. Also, we require that $p \alpha$ tends to infinity. This is valid by our assumption that $N\alpha^2$ tends to infinity as $N$ tends to infinity. Choose $s=\lfloor \frac{N}{p^2}\rfloor$ and notice that $k = O (p \alpha)$, then both $k$ and $s$ tend to infinity as $N$ tends to infinity.  
We have $p^2 s < N <p^2 (s+1)$ from the choice of $s$. From the conclusion for $p^2 s$ in Theorem \ref{Z}, we know for any $\alpha <1/2$, there exists a set $A$ such that $|A| \ge \alpha p^2s$ and $r_{A-A} \le  (1+o(1)) \alpha^2 p^2 s$. Regarding $A$ as a subset of $[N]$ in the natural way, we have $|A| \ge \alpha \times \frac{s}{s+1} N$ and $r_{A-A} \le (1+o(1)) \alpha^2 N$. Write $\alpha' = \alpha \times \frac{s}{s+1}$, then the result can be restated as 
\[ |A| \ge \alpha' N,  ~~r_{A-A} \le (1+o(1)) \left(\frac{s+1}{s} \right)^2  \alpha'^2 N. \]
The right hand side can be written as $(1+o(1)) \alpha'^2 N$ once we let $s$ tend to infinity.
The above fact  gives us the upper bound  in Theorem \ref{Lingbo 2}. Hence, we complete the proof.\\
\end{proof}

\bibliographystyle{plain}
\bibliography{xu}

\end{document}